\documentclass[3p,preprint]{elsarticle}
\usepackage[T1]{fontenc}
\usepackage{amsfonts}
\usepackage{amsmath}
\usepackage{geometry}
\usepackage{xcolor}

\newtheorem{theorem}{Theorem}[section]

\newtheorem{lemma}[theorem]{Lemma}

\newtheorem{remark}[theorem]{Remark}

\newenvironment{proof}[1][Proof]{\noindent\textbf{#1.} }{\ \rule{0.5em}{0.5em}}

\newcommand{\B}[1]{\boldsymbol{#1}}

\begin{document}

\bigskip

\bigskip
\begin{frontmatter}

\title{A new quadratic and cubic polynomial enrichment of the Crouzeix--Raviart finite element.}
\author[address-CS,address-CNR]{Francesco Dell'Accio\corref{corrauthor}}
\cortext[corrauthor]{Corresponding author}
 \ead{francesco.dellaccio@unical.it}

\author[address-Pau]{Allal Guessab}
\ead{allal.guessab@univ-pau.fr}
  \author[address-CS,address-Pau]{Federico Nudo}
  \ead{federico.nudo@unical.it}

  \address[address-CS]{Department of Mathematics and Computer Science, University of Calabria, Rende (CS), Italy}
  \address[address-Pau]{Laboratoire de Mathématiques et de leurs Applications, UMR CNRS 5142, Université de Pau et des Pays de l'Adour (UPPA), 64000 Pau, France}

  \address[address-CNR]{Istituto per le Applicazioni del Calcolo “Mauro Picone”, Naples Branch, C.N.R. National Research Council of Italy, Napoli, Italy}

\begin{abstract}
In this paper, we introduce quadratic and cubic polynomial enrichments of the classical Crouzeix--Raviart finite element, with the aim of constructing accurate approximations in such enriched elements. To achieve this goal,
we respectively add three and seven weighted line integrals as enriched degrees of freedom. For each case, we present a necessary and sufficient condition under which
these augmented elements are well-defined. For illustration purposes, we then use a general approach to define two-parameter  
 families of admissible degrees of freedom.
 Additionally, we provide
explicit expressions for the associated basis functions and subsequently introduce new quadratic and
cubic approximation operators based on the proposed admissible elements.
The efficiency of the enriched methods is compared to the triangular Crouzeix--Raviart element. As expected, the numerical results exhibit a significant improvement,
confirming the effectiveness of
the developed enrichment strategy.
\end{abstract}


\end{frontmatter}

\section{Introduction}
The analysis of various physical phenomena often requires the widespread application of the finite element method. This method is extensively used to approximate solutions of partial differential equations (PDEs), where the problem domain is divided into subdomains known as \textit{finite elements}. Each finite element corresponds to a local solution of the differential problem, and the collective contribution of these local solutions constructs a global piecewise solution. A finite element is classified as \textit{conforming} if the global piecewise solution remains continuous across subdomain boundaries; otherwise, it is classified as \textit{nonconforming}. To enhance the local approximation accuracy generated by the finite element method, a commonly employed technique is to augment the finite element with \textit{enrichment elements}~\cite{Guessab:2022:SAB, Guessab:2016:AADM, Guessab:2016:RM, DellAccio:2022:QFA, DellAccio:2022:AUE, DellAccio:2022:ESF, DellAccio2023AGC, DellAccio2023nuovo}. The Crouzeix--Raviart finite element, named after mathematicians Vivette Crouzeix and Pierre-Arnaud Raviart, is a particular finite element employed in numerical analysis to solve second-order elliptic PDEs. Its notable characteristic is its nonconforming nature, offering flexibility in handling complex geometries and irregular meshes~\cite{CR1973:2022:CR}. In the Crouzeix--Raviart finite element method, the approximation spaces are chosen to efficiently represent discontinuous solutions. This feature makes it well-suited for problems with solutions featuring singularities or discontinuities. The Crouzeix--Raviart finite element has found applications in various fields, including fluid dynamics, structural mechanics, and other areas where accurate and flexible numerical solutions of PDEs are crucial~\cite{chatzipantelidis1999finite, hansbo2003discontinuous, burman2005stabilized, zhu2014analysis, di2015extension, ve2019quasi}.

In this paper, we introduce two enrichments of the classical Crouzeix--Raviart finite element based on quadratic and cubic polynomials. These enrichments are realized by employing weighted line integrals as enriched degrees of freedom. 
The paper is organized as follows. In Section~\ref{s1}, we introduce a new enrichment of the classical Crouzeix--Raviart finite element using quadratic polynomial functions. Additionally, we derive an explicit expression for the associated basis functions and then introduce a new quadratic approximation operator based on this enriched finite element. In line with the previous section, in Section~\ref{sec3} we propose a cubic polynomial enrichment of the Crouzeix--Raviart finite element, and we introduce a new cubic approximation operator. Numerical results are provided in Section~\ref{sec4}.

\section{A quadratic polynomial enrichment of the Crouzeix--Raviart finite element}\label{s1}
Let $T \subset \mathbb{R}^2$ be a nondegenerate triangle with vertices $\B{v}_1, \B{v}_2, \B{v}_3$. For $i=1,2,3$, we denote by $\lambda_i$ the barycentric coordinate associated to the vertex $\B{v}_i$, and by $\Gamma_i$ the edge of $T$ opposite to the vertex $\B{v}_i$. These functions  are affine functions and satisfy the following properties
\begin{equation}\label{bary}
\lambda_i(\B{v}_j) = \delta_{ij}:=
\begin{cases}
\begin{aligned}
    &1, \quad \text{if } i=j, \\
    &0, \quad \text{if } i\neq j,
\end{aligned}
\end{cases}
\quad i,j=1,2,3,
\end{equation}
\begin{equation*}
\sum_{i=1}^3 \lambda_i(\B x) = 1, \qquad \B x\in T
\end{equation*}
and
\begin{equation}\label{prop2}
\lambda_i(t \B{x} + (1-t) \B{y}) = t \lambda_i(\B{x}) + (1-t)\lambda_i(\B{y}), \qquad \B x, \B y\in T, \qquad t\in[0,1].
\end{equation}
A direct consequence of~\eqref{bary} and~\eqref{prop2} is 
\begin{equation}\label{propstar}
  \lambda_i(\B x)=0, \quad \B x\in \Gamma_i, \quad i=1,2,3.  
\end{equation}
In the subsequent discussion, we will leverage a classical result that is applicable to any $d$-simplex $T_d$, where $d\in\mathbb{N}$~\cite[Chapter 2]{Dunkl:2014:OPO}.
\begin{lemma}\label{lem1old}
Let $\gamma_0,\ldots, \gamma_d$ be nonnegative real numbers. Then the following identity holds
\begin{equation}\label{idbc}
\frac{1}{\left\lvert T_d\right\rvert}\int_{T_d}\prod_{i=0}^{d}\lambda_i^{\gamma_i}(\B{x})\, d\B{x} = \frac{d! \prod_{i=0}^{d}\Gamma(\gamma_i+1) }{\Gamma(d+1+\sum_{i=0}^{d}\gamma_i)},
\end{equation}
where $\left\lvert T_d\right\rvert$ is the volume of $T_d$ and $\Gamma(z)$ is the gamma function~\cite{Abramowitz:1948:HOM}.
\end{lemma}
Throughout the paper, we consistently adopt the cyclic numbering convention $\B{v}_4 := \B{v}_1$ and $\B{v}_5 := \B{v}_2$, along with $ \lambda_4 := \lambda_1$ and $\lambda_5 := \lambda_2$.
For $i=1,2,3$, we denote by $\mathcal{I}_j^{\mathrm{CR}}$ the following linear functional 
\begin{equation}\label{lev}
 \mathcal{I}_j^{\mathrm{CR}}(f) :=\frac{1}{\left\lvert \Gamma_j\right\rvert} \int_{\Gamma_j} f(s)\, ds = \int_{0}^{1} f\left(t\B{v}_{j+1} + (1-t)\B{v}_{j+2}\right)\, dt, \quad j=1,2,3.
\end{equation}

The classical nonconforming Crouzeix--Raviart finite element (CR finite element) of the first order, introduced in~\cite{CR1973:2022:CR}, is locally defined as 
\begin{equation*}
CR := \left(T,  \mathbb{P}_1(T), \Sigma^{\mathrm{CR}}\right),    
\end{equation*}
where 
\begin{equation*}
     \mathbb{P}_1(T):=\operatorname{span}\{\lambda_1,\lambda_2,\lambda_3\}
\end{equation*}
is the space of bivariate linear polynomials on $T$ and 
\begin{equation*}
\Sigma^{\mathrm{CR}} := \left\{\mathcal{I}_j^{\mathrm{CR}} \, : \, j=1,2,3\right\}.     
\end{equation*}
Let $\alpha,\beta>-\frac{1}{2}$ be real numbers. Depending on these parameters, in this section, we introduce a quadratic polynomial enrichment of the finite element $CR$. To this aim, for each $j=1,2,3$, we define the following enriched linear functionals 
 \begin{equation}\label{ex3}
 \mathcal{F}_{j}^{\mathrm{enr}}(f):= \int_{0}^{1}w_\alpha(t)\left(2(\alpha+ \beta+1)\widetilde{L}_1^{2}(t)-(2\beta+1)\right)\left\lvert\widetilde{L}_1(t)\right\rvert^{2\beta}f\left(t\B{v}_{j+1}+ (1-t)\B{v}_{j+2}\right)\, dt,
\end{equation} 
where 
\begin{equation}\label{LegPol}
\widetilde{L}_1(t):=2t-1    
\end{equation}
is the shifted Legendre polynomial of degree one and 
\begin{equation}
    \label{weight}
w_\alpha(t):= t^{\alpha-\frac{1}{2}} (1 - t)^{\alpha-\frac{1}{2}}    
\end{equation}
is the weight function.
We define the following triple 
\begin{equation}\label{P2CRfinielement}
    \mathcal{C}:= \left(T,   \mathbb{P}_2(T),\Sigma_{2, T}^{\mathrm{enr}}\right),
\end{equation}
where   
 \begin{equation*}
 \Sigma_{2,T}^{\mathrm{enr}}:=  \left\{\mathcal{I}_j^{\mathrm{CR}},\mathcal{F}^{\mathrm{enr}}_{j}\, : \, j=1,2,3\right\}.
\end{equation*}

\begin{remark}
    We remark that, by using the change of variable  $u = 2t-1$, the enriched linear functionals $\mathcal{F}_{j}^{\mathrm{enr}}$, $j=1,2,3,$ can be rewritten as follows
  \begin{equation}\label{ex3b}
 \mathcal{F}_{j}^{\mathrm{enr}}(f) = \frac{1}{2^{2\alpha}} \int_{-1}^{1}\widetilde{w}_{\alpha,\beta}(u)p_2(u)f\left(\frac{1+u}{2}\B{v}_{j+1}+ \frac{1-u}{2}\B{v}_{j+2}\right)\, du, \qquad j=1,2,3,
\end{equation} 
where 
\begin{equation}\label{p2orth}
  p_2(u):= 2(\alpha+ \beta+1)u^2-(2\beta+1)
\end{equation}
and
\begin{equation}\label{newweight}
\widetilde{w}_{\alpha,\beta}(u):= (1-u^2)^{\alpha-\frac{1}{2}}\left\lvert u\right\rvert^{2\beta}.
\end{equation}
Moreover, we observe that the weight $\widetilde{w}_{\alpha,\beta}(u)$ is the product of the classical Gegenbauer (ultraspherical) weight function with $\left\lvert u\right\rvert^{2\beta}$.
\end{remark}

\begin{remark}
    Varying the parameters $\alpha$ and $\beta$, the polynomial $p_2$ defined in~\eqref{p2orth} used to define the enriched degrees of freedom~\eqref{ex3b} allow to cover a wide class of classical orthogonal polynomials. For example, up to scaling, we get
\begin{itemize}
    \item if $\alpha= \frac{1}{2}, \beta=0$,  $p_2(u)=3u^2-1$ is the 2-th Legendre polynomial;
    \item if $\alpha>\frac{1}{2}$, $\beta=0$,  $p_2(u)=2(\alpha+1)u^2-1$ is the 2-th Gegenbauer polynomial;
    \item if $\alpha= \beta$,  $\beta > 0$, $p_2=(2\beta+1)\left(2 u^2-1 \right)$ is the 2-th Chebyshev polynomial of  the first kind;
    \item if $\alpha= 3\beta+1,$ $\beta > 0$, $ p_2(u)=(2\beta+1)\left(4 u^2-1 \right)$ is the 2-th Chebyshev polynomial of  the second kind.
\end{itemize}
It should be mentioned that, in the last two cases, the corresponding weight functions $\widetilde{w}_{\alpha,\beta}(u)$ are not the classical Jacobi weights.
\end{remark}

We aim to demonstrate that the triple $\mathcal{C}$ defined in~\eqref{P2CRfinielement} is a finite element. To achieve this, the subsequent lemmas hold fundamental importance. 

\begin{lemma}\label{lem1}
The polynomial $p_2$ defined in~\eqref{p2orth} is orthogonal to all linear polynomials on the interval $[-1, 1]$ with respect to the weight function $\widetilde{w}_{\alpha,\beta}.$ In other words, for any linear polynomial $p\in  \mathbb{P}_1(T),$ the following equality holds
\begin{equation*}
\int_{-1}^{1}\widetilde{w}_{\alpha,\beta}(u) p_2(u)p(u)\, du = 0.
\end{equation*}
\end{lemma}
\begin{proof}
Since the integral of an odd function over a symmetric domain is equal to zero, it suffices to show that
 \begin{equation}\label{orth2}
\int_{-1}^{1}\widetilde{w}_{\alpha,\beta}(u) p_2(u) \, du =2 \int_{0}^{1}\widetilde{w}_{\alpha,\beta}(u) p_2(u) \, du =0.
\end{equation}
By~\eqref{p2orth} and~\eqref{newweight} and by using the change of variable $v= u^2$, we get
\begin{eqnarray}
&&2\int_{0}^{1}\widetilde{w}_{\alpha,\beta}(u) p_2(u) \, du = \notag \\ &=& 4(\alpha + \beta+1)\int_{0}^{1} u^{2\beta+2}(1-u^2)^{\alpha-1/2} du -2(2\beta+1)\int_{0}^{1}u^{2\beta}(1-u^2)^{\alpha-1/2} du \notag \\
&=& 2(\alpha + \beta+1)\int_{0}^{1} v^{\beta+\frac{1}{2}}(1-v)^{\alpha-\frac{1}{2}} dv -(2\beta+1)\int_{0}^{1}v^{\beta-\frac{1}{2}}(1-v)^{\alpha-\frac{1}{2}} dv \notag\\
&=&2(\alpha + \beta+1)B\left(\beta+\frac{3}{2},\alpha+\frac{1}{2}\right)-  (2\beta+1)B\left(\beta+\frac{1}{2},\alpha+\frac{1}{2}\right) \label{eq122}
\end{eqnarray}
where 
\begin{equation*}
    B(z_1,z_2)=\int_{0}^{1}u^{z_1-1}(1-u)^{z_2-1} du, \qquad z_1,z_2>-1,
\end{equation*}
is the classical Euler beta function~\cite{Abramowitz:1948:HOM}.
It is well known (see, e.g., \cite{Abramowitz:1948:HOM}) that this function satisfies
\begin{equation}\label{propbetafun}
B(z_1+1,z_2)=\frac{z_1}{z_1+z_2}B(z_1,z_2), \qquad B(z_1,z_2+1)=\frac{z_2}{z_1+z_2}B(z_1,z_2), \qquad z_1,z_2>-1.  
\end{equation}
Then, since 
\begin{equation}\label{orths}
    B\left(\beta+\frac{3}{2},\alpha+\frac{1}{2}\right)= \frac{2\beta+1}{ 2(\alpha + \beta+1)}B\left(\beta+\frac{1}{2},\alpha+\frac{1}{2}\right),
\end{equation}
the orthogonality relation~\eqref{orth2} follows by substituting~\eqref{orths} in~\eqref{eq122}. 
\end{proof}

The next result is a direct application of Lemma~\ref{lem1}.

\begin{lemma}\label{adomnnew}
  For any $p\in  \mathbb{P}_1(T)$, the following equality holds
 \begin{equation*}
\mathcal{F}_{j}^{{\mathrm{enr}}}(p) = 0, \qquad j=1,2,3.
 \end{equation*} 
\end{lemma} 
\begin{proof}
 Let $p\in  \mathbb{P}_1(T)$. For each $j=1,2,3,$ its  restriction to $\Gamma_j$
 \begin{equation*}
     p(t\B v_{j+1}+ (1-t)\B v_{j+2}), \quad t\in[0,1], 
 \end{equation*}
is a linear polynomial in the variable $t$. The statement now follows trivially from Lemma~\ref{lem1}.
\end{proof}

In the following theorem, we prove that the triple  $\mathcal{C}$ defined in~\eqref{P2CRfinielement} is a finite element. 

\begin{theorem}\label{th1nnewfin1} 
  For any real numbers $\alpha,\beta>-\frac{1}{2}$ the triple $ \mathcal{C}$ defined in~\eqref{P2CRfinielement} is a finite element.
\end{theorem}

\begin{proof}
 Let $p \in   \mathbb{P}_2(T)$ such that 
\begin{eqnarray}
  \mathcal{I}_j^{\mathrm{CR}}(p) &=& 0, \qquad j=1,2,3, \label{ho1}\\
  \mathcal{F}_{j}^{\mathrm{enr}}(p) &=& 0, \qquad j=1,2,3. \label{ho21}
\end{eqnarray}
We must prove that $p=0$. Since  $p\in  \mathbb{P}_2(T)$, it  can be expressed as   
\begin{equation*}
  p= p_1+ a\lambda_1^2+ b \lambda_2^2+c\lambda_3^2,
\end{equation*}
where  $p_1 \in   \mathbb{P}_1(T)$ and $a,b,c \in \mathbb{R}.$ Considering~\eqref{propstar} and Lemma~\ref{lem1}, we obtain
\begin{eqnarray}
0 &=&\mathcal{F}_{1}^{\mathrm{enr}}(p) = b \mathcal{F}_{1}^{\mathrm{enr}}(\lambda_2^2)+c \mathcal{F}_{1}^{\mathrm{enr}}(\lambda_3^2) ,\label{genn1new1} \\ 
0&=& \mathcal{F}_{2}^{\mathrm{enr}}(p) = a \mathcal{F}_{2}^{\mathrm{enr}}(\lambda_1^2)+ c \mathcal{F}_{2}^{\mathrm{enr}}(\lambda_3^2),\label{genn2new2} \\ 
0&=&\mathcal{F}_{3}^{\mathrm{enr}}(p) =  a \mathcal{F}_{3}^{\mathrm{enr}}(\lambda_1^2) + b \mathcal{F}_{3}^{\mathrm{enr}}(\lambda_2^2).\label{genn3new3} 
\end{eqnarray}
Using properties~\eqref{bary} and~\eqref{prop2}, along with straightforward computations, we obtain
\begin{equation*}
\mathcal{F}_{1}^{\mathrm{enr}}(\lambda_2^2)=\mathcal{F}_{1}^{\mathrm{enr}}(\lambda_3^2)=\mathcal{F}_{2}^{\mathrm{enr}}(\lambda_1^2)=\mathcal{F}_{2}^{\mathrm{enr}}(\lambda_3^2)=
\mathcal{F}_{3}^{\mathrm{enr}}(\lambda_1^2)=\mathcal{F}_{3}^{\mathrm{enr}}(\lambda_2^2)=
K,
\end{equation*}
where
\begin{equation}\label{denot}
  K:=\frac{1}{2^{2\alpha+1}}B\left(\beta+\frac{3}{2},\alpha+\frac{3}{2}\right).
\end{equation}
Therefore, we can rewrite~\eqref{genn1new1}, \eqref{genn2new2} and~\eqref{genn3new3} in matrix form as follows
\begin{equation*}
    \begin{bmatrix}
0 &  K & K\\
K & 0& K \\
K &K &0
\end{bmatrix}
 \begin{bmatrix}
a\\
b\\
c
\end{bmatrix}
= \begin{bmatrix}
0\\
0\\
0
\end{bmatrix}.
 \end{equation*}
The matrix associated with the linear system has determinant $2K^3 \neq 0$. Therefore, the linear system has the unique solution $a=b=c=0$. Consequently, $p=p_1$ is a linear polynomial. The statement follows from the CR element, implying that $p_1=0$.
\end{proof}

\begin{remark}
    From the previous theorem, it can be observed that the nonsingularity of the matrix 
    \begin{equation*}
      \begin{bmatrix}
\mathcal{F}_{1}^{\mathrm{enr}}(\lambda_1^2) &  \mathcal{F}_{1}^{\mathrm{enr}}(\lambda_2^2) & \mathcal{F}_{1}^{\mathrm{enr}}(\lambda_3^2)\\
\mathcal{F}_{2}^{\mathrm{enr}}(\lambda_1^2) & \mathcal{F}_{2}^{\mathrm{enr}}(\lambda_2^2)& \mathcal{F}_{2}^{\mathrm{enr}}(\lambda_3^2) \\
\mathcal{F}_{3}^{\mathrm{enr}}(\lambda_1^2) & \mathcal{F}_{3}^{\mathrm{enr}}(\lambda_2^2)& \mathcal{F}_{3}^{\mathrm{enr}}(\lambda_3^2) \\
\end{bmatrix}=
     \begin{bmatrix}
0 &  K & K\\
K & 0& K \\
K &K &0
\end{bmatrix}
\end{equation*}
    is a necessary and sufficient condition for $\mathcal{C}$ to be a finite element.
\end{remark}

As a direct consequence of Theorem~\ref{th1nnewfin1}, we can establish the existence of a basis
\begin{equation*}
    \mathcal{B}_2=\{\varphi_i, \phi_i\, :\, i=1,2,3\}
\end{equation*}
of the polynomial space $\mathbb{P}_2(T)$ satisfying 
the following conditions
\begin{eqnarray}
\label{propvarphi}
&& \mathcal{I}_j^{\mathrm{CR}}(\varphi_i) = \delta_{ij}, \quad \mathcal{F}^{\mathrm{enr}}_j(\varphi_i) = 0, \quad i,j=1,2,3, \\ \label{propphi}
&&\mathcal{I}_j^{\mathrm{CR}}(\phi_i) = 0, \quad \mathcal{F}^{\mathrm{enr}}_j(\phi_i) = \delta_{ij}, \quad i,j=1,2,3.
\end{eqnarray}
Here, $\delta_{ij}$ represents the Kronecker delta symbol. These functions are commonly referred to as the basis functions of the enriched finite element $\mathcal{C}$.
In the next Theorem, we give simple, closed-form expressions of the basis functions of $ \mathcal{C}.$
\begin{theorem}\label{th2allalf} The basis functions of the enriched finite element $\mathcal{C}$ have the following expressions
\begin{eqnarray}
    \varphi_i&=& 1-2\lambda_i,  \qquad i=1,2,3 \label{varphi}\\
    \phi_i&=&-\frac{\varphi_i}{3K}+\frac{1}{2K} (-\lambda_{i}^2 +\lambda_{i+1}^2+\lambda_{i+2}^2), \qquad i=1,2,3, \label{phi}
\end{eqnarray}
where the normalized coefficient $K$ is defined in~\eqref{denot}. 
      \end{theorem}
\begin{proof}
Without loss on generality, we prove the theorem for $i=1$. 

Since $\varphi_1\in \mathbb{P}_2(T)$, it can be expressed as
\begin{equation}\label{phi1old}
\varphi_1 = \sum_{i=1}^3 a_{1,i}\lambda_i + \sum_{k=1}^3 b_{1,k} \lambda_k^2,
\end{equation}
where $a_{1,i}, b_{1,k}\in \mathbb{R}$, $i,k=1,2,3$.
Applying the linear functional $\mathcal{F}^{\mathrm{enr}}_j$, $j=1,2,3$, to both sides of~\eqref{phi1old} and leveraging Lemma~\ref{adomnnew} and~\eqref{propvarphi}, we derive the linear system
\begin{equation*}
0 = \mathcal{F}^{\mathrm{enr}}_j(\varphi_1) = \sum_{\substack{k=1}}^{3} b_{1,k} \mathcal{F}^{\mathrm{enr}}_j(\lambda^2_k)=\sum_{\substack{k=1 \\ k \neq j}}^{3} b_{1,k} K, \qquad j=1,2,3,
\end{equation*}
where $\mathcal{F}^{\mathrm{enr}}_j(\lambda^2_j)=0$ due to~\eqref{propstar}.
Expressing this system more compactly, we get
\begin{equation*}
  \begin{bmatrix}
0 &  K & K\\
K & 0& K \\
K &K &0
\end{bmatrix}
\begin{bmatrix}
b_{1,1}\\
b_{1,2}\\
b_{1,3}\\
\end{bmatrix}
= \begin{bmatrix}
0\\
0\\
0\\
\end{bmatrix}.
\end{equation*}
Since the associated matrix is nonsingular, we conclude that 
\begin{equation*}
b_{1,1} = b_{1,2} = b_{1,3} = 0.
\end{equation*}
Substituting these values into~\eqref{phi1old}, we obtain
\begin{equation*}
\varphi_1 = \sum_{i=1}^3 a_{1,i}\lambda_i.
\end{equation*}
By~\eqref{propvarphi},~\eqref{idbc} and~\eqref{prop2}, we have
\begin{eqnarray*}
1 &=& \mathcal{I}^{\mathrm{CR}}_{1}(\varphi_1) = \frac{a_{1,2}+a_{1,3}}{2} \\
0 &=& \mathcal{I}^{\mathrm{CR}}_{2}(\varphi_1) = \frac{a_{1,3}+a_{1,1}}{2} \\
0 &=& \mathcal{I}^{\mathrm{CR}}_{3}(\varphi_1) = \frac{a_{1,1}+a_{1,2}}{2}.
\end{eqnarray*}
This leads to the expression
\begin{equation*}
\varphi_1 = 1-2\lambda_1.
\end{equation*}

It remains to prove~\eqref{phi}. Since $\phi_1\in \mathbb{P}_2(T)$, it can be expressed as
\begin{equation}\label{phi1}
\phi_1 = \sum_{i=1}^3 c_{1,i}\lambda_i + \sum_{k=1}^3 d_{1,k} \lambda_k^2.
\end{equation}
Applying the linear functional $\mathcal{F}^{\mathrm{enr}}_j$, $j=1,2,3,$ to both sides of~\eqref{phi1}, and leveraging Lemma~\ref{adomnnew} and~\eqref{propvarphi}, we derive the linear system

\begin{equation*} 
\delta_{1j} = \mathcal{F}^{\mathrm{enr}}_j(\phi_1) =\sum_{\substack{k=1}}^{3} d_{1,k} \mathcal{F}^{\mathrm{enr}}_j(\lambda_k^2)= \sum_{\substack{k=1 \\ k \neq j}}^{3} d_{1,k} K, \qquad j=1,2,3.
\end{equation*}
Expressing this system more compactly, we have
\begin{equation*}
  \begin{bmatrix}
0 &  K & K\\
K & 0& K \\
K &K &0
\end{bmatrix}
\begin{bmatrix}
d_{1,1}\\
d_{1,2}\\
d_{1,3}\\
\end{bmatrix}
= \begin{bmatrix}
1\\
0\\
0\\
\end{bmatrix}.
\end{equation*}
Then, we get 
\begin{equation*}
 \frac{1}{2K^3} \begin{bmatrix}
-K^2 &  K^2 & K^2\\
K^2 & -K^2& K^2 \\
K^2 &K^2 &-K^2
\end{bmatrix}
\begin{bmatrix}
1\\
0\\
0\\
\end{bmatrix}
=\begin{bmatrix}
d_{1,1}\\
d_{1,2}\\
d_{1,3}\\
\end{bmatrix},
\end{equation*}
and therefore
\begin{equation*}
    d_{1,1}=-\frac{1}{2K}, \quad d_{1,2}=\frac{1}{2K}, \quad d_{1,3}=\frac{1}{2K}.
\end{equation*}
Substituting these values into~\eqref{phi1}, we obtain 
\begin{equation*} 
    \phi_1=\sum_{i=1}^3 c_{1,i} \lambda_i +\frac{1}{2K}(-\lambda_1^2+\lambda^2_2+\lambda^2_3). 
\end{equation*}
By~\eqref{propphi},~\eqref{idbc} and~\eqref{prop2}, we get
\begin{eqnarray*}
0 &=& \mathcal{I}^{\mathrm{CR}}_{1}(\phi_1) = \frac{c_{1,2}+c_{1,3}}{2}+\frac{1}{3K} \\
0 &=& \mathcal{I}^{\mathrm{CR}}_{2}(\phi_1) = \frac{c_{1,3}+c_{1,1}}{2} \\
0 &=& \mathcal{I}^{\mathrm{CR}}_{3}(\phi_1) = \frac{c_{1,1}+c_{1,2}}{2}.
\end{eqnarray*}
Then, we have
\begin{equation*}
    c_{1,1}=\frac{1}{3K}, \quad c_{1,2}=-\frac{1}{3K}, \quad c_{1,3}= -\frac{1}{3K}. 
\end{equation*}
This leads to the expression
\begin{equation*}
    \phi_1= -\frac{\varphi_1}{3K}+\frac{1}{2K}(-\lambda_1^2+\lambda^2_2+\lambda^2_3).
\end{equation*}
Analogously, the theorem can be proven for $i=2$ and $i=3$, and thus, the thesis follows.
\end{proof}

\begin{remark}
We observe that the basis functions $\varphi_i$, $i=1,2,3$, defined in~\eqref{varphi}, correspond to the Crouzeix–Raviart nonconforming basis functions, see~\cite{Guessab:2016:AADM}.
\end{remark}

\begin{theorem}
 The approximation operator relative to the enriched finite element $\mathcal{C}$
\begin{equation}\label{pilinch9C}
\begin{array}{rcl}
{\Pi}_2^{{\mathrm{enr}}}: C(T) &\rightarrow& \mathbb{P}_2(T)
\\
f &\mapsto& \displaystyle{\sum_{j=1}^{3}  \mathcal{I}^{\mathrm{CR}}_j(f)\varphi_j+ \sum_{j=1}^{3}\mathcal{F}^{\mathrm{enr}}_j(f)}\phi_j,
\end{array}
\end{equation}
reproduces all polynomials of $\mathbb{P}_2(T)$ and satisfies 
\begin{eqnarray*}
\mathcal{I}^{\mathrm{CR}}_j\left({\Pi}_2^{\mathrm{enr}}[f]\right)&=&\mathcal{I}^{\mathrm{CR}}_j(f), \qquad j=1,2,3, \\ \mathcal{F}^{\mathrm{enr}}_j\left({\Pi}_2^{\mathrm{enr}}[f]\right)&=&\mathcal{F}^{\mathrm{enr}}_j(f), \qquad j=1,2,3.
\end{eqnarray*} 
\end{theorem}
\begin{proof}
The proof is a consequence of~\eqref{propvarphi} and~\eqref{propphi}.
\end{proof}

\section{A cubic polynomial enrichment of the Crouzeix--Raviart finite element}\label{sec3}
Let $\alpha, \beta > -\frac{1}{2}$ be real numbers. In this section, we introduce a cubic polynomial enrichment of the Crouzeix--Raviart finite element, based on the aforementioned parameters. Following a similar approach as in the previous section, we define the following enriched linear functionals for each $j=1,2,3$
\begin{align*}
     \mathcal{L}_{j}^{{\mathrm{enr}}}(f) &:= \int_{0}^{1}w_\alpha(t) \left(2(\alpha + \beta + 2)\widetilde{L}_1^{2}(t) - (2\beta + 3)\right)\left\lvert \widetilde{L}_1(t)\right\rvert^{2\beta}\widetilde{L}_1(t)f\left(t\B{v}_{j+1} + (1-t)\B{v}_{j+2}\right)\, dt, \\
      \mathcal{J}^{{\mathrm{enr}}}(f) &:= \int_{T} f(\B{x})\, d\B{x},
\end{align*}
where $\widetilde{L}_1(t)$ and $w_\alpha(t)$ are defined in~\eqref{LegPol} and~\eqref{weight}, respectively. We define the triple 
\begin{equation}\label{tripleS}
\mathcal{S}:= (T, \mathbb{P}_3(T), \Sigma_{3,T}^{{\mathrm{enr}}}),    
\end{equation}
with 
 \begin{equation*}
 \Sigma_{3,T}^{{\mathrm{enr}}}:=  \left\{ \mathcal{I}_j^{\mathrm{CR}},\mathcal{F}^{\mathrm{enr}}_{j},\mathcal{L}^{\mathrm{enr}}_{j}, \mathcal{J}^{\mathrm{enr}} \, : \, j=1,2, 3 \right\}
\end{equation*}
where $\mathcal{I}_j^{\mathrm{CR}},\mathcal{F}^{\mathrm{enr}}_{j}$ $j=1,2,3,$ are defined as in~\eqref{lev} and~\eqref{ex3}, respectively.

\begin{remark}
    We remark that, by using the change of variable  $u = 2t-1$, the enriched linear functionals $ \mathcal{L}_{j}^{{\mathrm{enr}}}$, $j=1,2,3,$ can be rewritten as follows
 \begin{equation*}
 \mathcal{L}_{j}^{{\mathrm{enr}}}(f) = \frac{1}{2^{2\alpha}} \int_{-1}^{1}\widetilde{w}_{\alpha,\beta}(u)uq_2(u)f\left(\frac{1+u}{2}\B v_{j+1}+ \frac{1-u}{2}\B v_{j+2}\right)\, du, \qquad j=1,2,3,
\end{equation*}  
where
 \begin{equation}\label{orth2z}
   q_2(u) :=  2(\alpha+ \beta+2)u^2-(2\beta+3)
 \end{equation}
 and $ \widetilde{w}_{\alpha,\beta}$ is defined in~\eqref{newweight}.
\end{remark}

We aim to demonstrate that the triple $\mathcal{S}$ defined in~\eqref{tripleS} is a finite element. The subsequent lemmas play a crucial role in achieving this goal. In analogy to Lemma~\ref{lem1}, the following lemma establishes the orthogonality of the polynomial $uq_2$ within the space of quadratic polynomials on the interval $[-1, 1]$ with respect to the weight function $\widetilde{w}_{\alpha,\beta}$.

\begin{lemma}\label{lem1gennew}
The polynomial $uq_2$ is orthogonal to all quadratic polynomials on the interval $[-1, 1]$ with respect to the weight function $\widetilde{w}_{\alpha,\beta}.$ In other words, for any quadratic polynomial $p\in  \mathbb{P}_2(T),$ the following equality holds
  \begin{equation*}
   \int_{-1}^{1}\widetilde{w}_{\alpha,\beta}(u) uq_2(u)p(u)\, du =0.
\end{equation*} 
\end{lemma}
\begin{proof}
Since the integral of an odd function over a symmetric domain is equal to zero, we get
 \begin{equation*}
\int_{-1}^{1}\widetilde{w}_{\alpha,\beta}(u) u^3 q_2(u) \, du= \int_{-1}^{1}\widetilde{w}_{\alpha,\beta}(u) u q_2(u) \, du =0.
\end{equation*}
Then, it is sufficient to prove that 
    \begin{equation*}
\int_{-1}^{1}\widetilde{w}_{\alpha,\beta}(u) u^2 q_2(u) \, du =2\int_{0}^{1}\widetilde{w}_{\alpha,\beta}(u) u^2 q_2(u) \, du=0.
\end{equation*}
By~\eqref{newweight} and~\eqref{orth2z} and by using the change of variable $v= u^2$, we get
\begin{eqnarray*}
&&2\int_{0}^{1}\widetilde{w}_{\alpha,\beta}(u) u^2 q_2(u) \, du=2\int_{0}^{1}(1-u^2)^{\alpha-\frac{1}{2}}u^{2\beta+2} \left( 2(\alpha+ \beta+2)u^2-(2\beta+3)\right) \, du\\
    &=& 4(\alpha+ \beta+2)\int_{0}^{1}(1-u^2)^{\alpha-\frac{1}{2}}u^{2\beta+4} \, du -2(2\beta+3) \int_{0}^{1}(1-u^2)^{\alpha-\frac{1}{2}}u^{2\beta+2}\, du\\
    &=&2(\alpha+ \beta+2)\int_{0}^{1}(1-v)^{\alpha-\frac{1}{2}}v^{\beta+\frac{3}{2}} \, dv -(2\beta+3) \int_{0}^{1}(1-v)^{\alpha-\frac{1}{2}}v^{\beta+\frac{1}{2}}\, dv\\
    &=& 2(\alpha+ \beta+2)B\left(\beta+\frac{5}{2},\alpha+\frac{1}{2}\right) -(2\beta+3) B\left(\beta+\frac{3}{2},\alpha+\frac{1}{2}\right).
\end{eqnarray*}
The thesis is a consequence of the property~\eqref{propbetafun}.
\end{proof}

The next result is a direct application of Lemma~\ref{lem1gennew}.

\begin{lemma}\label{adomnnewgennew}
   For any $p\in  \mathbb{P}_2(T)$, the following equality holds 
\begin{equation*}
    \mathcal{L}_{j}^{{\mathrm{enr}}}(p) = 0.
\end{equation*}
\end{lemma} 
\begin{proof}
Let $p\in  \mathbb{P}_2(T)$. For each $j=1,2,3,$ its  restriction to $\Gamma_j$
 \begin{equation*}
     p(t\B v_{j+1}+ (1-t)\B v_{j+2}), \quad t\in[0,1], 
 \end{equation*}
is a quadratic polynomial in the variable $t$. The statement now follows trivially from Lemma~\ref{lem1gennew}.
\end{proof}

The next Theorem establishes that the triple  $\mathcal{S}$ defined in~\eqref{tripleS} is a finite element. 
\begin{theorem}\label{th1nnewfin12} 
  For any real numbers $\alpha,\beta>-\frac{1}{2}$ the triple $ \mathcal{S}$ defined in~\eqref{tripleS} is a finite element.
\end{theorem}
\begin{proof}
 Let  $p \in  \mathbb{P}_3(T)$ such that 
\begin{eqnarray}
  \mathcal{I}_j^{\mathrm{CR}}(p) &=& 0, \qquad j=1,2,3, \label{ho11}\\
  \mathcal{F}_{j}^{\mathrm{enr}}(p) &=& 0, \qquad j=1,2,3,\label{ho2} \\ 
  \mathcal{L}_{j}^{\mathrm{enr}}(p) &=& 0, \qquad j=1,2,3,\label{ho3}\\ 
\mathcal{J}^{{\mathrm{enr}}}(p) &=& 0. \label{ho4}
\end{eqnarray}
We must prove that $p=0$. Since  $p\in \mathbb{P}_3(T)$, it can be expressed as   
\begin{equation}\label{polsecpart}
  p= p_2+ a\lambda_2\lambda_3^2+ b \lambda_3\lambda_1^2+c\lambda_1\lambda_2^2 + d\lambda_1\lambda_2\lambda_3,
\end{equation}
with  $p_2 \in  \mathbb{P}_2(T)$ and $a,b,c,d \in \mathbb{R}.$ Considering~\eqref{propstar} and Lemma~\ref{adomnnewgennew}, we obtain
\begin{eqnarray}
0 &=&\mathcal{L}^{\mathrm{enr}}_{1}(p) = a \mathcal{L}^{\mathrm{enr}}_{1}(\lambda_2\lambda_3^2) ,\label{genn1new1L} \\ 
0&=& \mathcal{L}^{\mathrm{enr}}_{2}(p) = b \mathcal{L}^{\mathrm{enr}}_{2}(\lambda_3\lambda_1^2),\label{genn2new2L} \\ 
0&=&\mathcal{L}^{\mathrm{enr}}_{3}(p) =  c\mathcal{L}^{\mathrm{enr}}_{3}(\lambda_1\lambda_2^2).\label{genn3new3L} 
\end{eqnarray}
By employing properties~\eqref{bary} and~\eqref{prop2}, along with straightforward computations, we derive the following equalities
\begin{equation*}
\mathcal{L}^{\mathrm{enr}}_{1}(\lambda_2\lambda_3^2)=\mathcal{L}^{\mathrm{enr}}_{2}(\lambda_3\lambda_1^2)=\mathcal{L}^{\mathrm{enr}}_{3}(\lambda_1\lambda_2^2)=:G\neq 0, 
\end{equation*}
where
\begin{equation}\label{kprimo}
  G=\frac{2\beta +3}{4(\alpha + \beta + 3)} K
\end{equation}
and $K$ is the constant defined in~\eqref{denot}. Therefore, we can rewrite~\eqref{genn1new1L}, \eqref{genn2new2L} and~\eqref{genn3new3L} in matrix form as follows
\begin{equation*}
    \begin{bmatrix}
G &  0 & 0\\
0 & G & 0 \\
0 & 0 & G
\end{bmatrix}
 \begin{bmatrix}
a\\
b\\
c
\end{bmatrix}
= \begin{bmatrix}
0\\
0\\
0
\end{bmatrix}.
 \end{equation*}
The matrix associated with the linear system has a nonzero determinant. Therefore, we conclude that $a=b=c=0$, and as a consequence, 
\begin{equation}\label{eeeee}
p=p_2+d\lambda_1\lambda_2\lambda_3.
\end{equation}
Since, by~\eqref{propstar}, we have
\begin{equation*} \mathcal{I}^{\mathrm{CR}}_j(\lambda_1\lambda_2\lambda_3)=\mathcal{F}^{\mathrm{enr}}_{j}(\lambda_1\lambda_2\lambda_3)=0, \qquad j=1,2,3,
\end{equation*}
conditions~\eqref{ho11} and~\eqref{ho2} can be expressed as
\begin{equation*}
0=\mathcal{I}^{\mathrm{CR}}_j(p)=\mathcal{I}^{\mathrm{CR}}_j(p_2) \qquad j=1,2,3,
\end{equation*}
\begin{equation*}
   0=\mathcal{F}^{\mathrm{enr}}_{j}(p)=\mathcal{F}^{\mathrm{enr}}_{j}(p_2), \qquad j=1,2,3.
\end{equation*}
According to Theorem~\ref{th1nnewfin1}, we can infer that $p_2=0$, and consequently,~\eqref{eeeee} is reduced to
\begin{equation*}
    p=d\lambda_1\lambda_2\lambda_3. 
\end{equation*}
Finally, by using~\eqref{idbc} of Lemma~\ref{lem1old}, we have
\begin{equation*}
    0=\mathcal{J}^{{\mathrm{enr}}}(p)=d\mathcal{J}^{{\mathrm{enr}}}(\lambda_1\lambda_2\lambda_3)=d  \frac{\left\lvert T \right\rvert}{60},
\end{equation*}
and then $d=0$. 
The theorem is proved. 
\end{proof}

\begin{remark}
    From the previous theorem, it can be observed that the nonsingularity of the matrix 
    \begin{equation*}
      \begin{bmatrix}
\mathcal{L}_{1}^{\mathrm{enr}}(\lambda_2\lambda_3^2) &  \mathcal{L}_{1}^{\mathrm{enr}}(\lambda_3\lambda_1^2) & \mathcal{L}_{1}^{\mathrm{enr}}(\lambda_1\lambda_2^2)\\
\mathcal{L}_{2}^{\mathrm{enr}}(\lambda_2\lambda_3^2) &  \mathcal{L}_{2}^{\mathrm{enr}}(\lambda_3\lambda_1^2) & \mathcal{L}_{2}^{\mathrm{enr}}(\lambda_1\lambda_2^2)\\
\mathcal{L}_{3}^{\mathrm{enr}}(\lambda_2\lambda_3^2) &  \mathcal{L}_{3}^{\mathrm{enr}}(\lambda_3\lambda_1^2) & \mathcal{L}_{3}^{\mathrm{enr}}(\lambda_1\lambda_2^2)\\
\end{bmatrix}=
     \begin{bmatrix}
G &  0 & 0\\
0 & G& 0 \\
0 &0 &G
\end{bmatrix}
\end{equation*}
    is a necessary and sufficient condition for $\mathcal{S}$ to be a finite element.
\end{remark}

As a direct consequence of Theorem~\ref{th1nnewfin12}, we can establish the existence of a basis
\begin{equation*}
    \mathcal{B}_3=\{\mu_i, \eta_i, \zeta_i, \xi\, :\, i=1,2,3\}
\end{equation*}
of the polynomial space $\mathbb{P}_3(T)$ satisfying 
the following conditions
\begin{eqnarray}
    & &\mathcal{I}_j^{\mathrm{CR}}(\mu_i) = \delta_{ij}, \quad \mathcal{F}^{\mathrm{enr}}_j(\mu_i) = 0, \quad \mathcal{L}^{\mathrm{enr}}_j(\mu_i)=0, \quad \mathcal{J}^{\mathrm{enr}}(\mu_i)=0, \quad i,j=1,2,3, \label{cond1}\\
      & &\mathcal{I}_j^{\mathrm{CR}}(\eta_i) = 0, \quad \mathcal{F}^{\mathrm{enr}}_j(\eta_i) = \delta_{ij}, \quad \mathcal{L}^{\mathrm{enr}}_j(\eta_i)=0, \quad \mathcal{J}^{\mathrm{enr}}(\eta_i)=0, \quad i,j=1,2,3, \label{cond2}\\
        & &\mathcal{I}_j^{\mathrm{CR}}(\zeta_i) = 0, \quad \mathcal{F}^{\mathrm{enr}}_j(\zeta_i) = 0, \quad \mathcal{L}^{\mathrm{enr}}_j(\zeta_i)=\delta_{ij}, \quad \mathcal{J}^{\mathrm{enr}}(\zeta_i)=0, \quad i,j=1,2,3, \label{cond3}\\
          & &\mathcal{I}_j^{\mathrm{CR}}(\xi) = 0, \quad \mathcal{F}^{\mathrm{enr}}_j(\xi) = 0, \quad \mathcal{L}^{\mathrm{enr}}_j(\xi)=0, \quad \mathcal{J}^{\mathrm{enr}}(\xi)=1, \quad j=1,2,3, \label{cond4}
\end{eqnarray}
where $\delta_{ij}$ represents the Kronecker delta symbol. These functions are commonly referred to as the basis functions of the enriched finite element $\mathcal{S}$.
In the next theorem, we give simple, closed-form expressions of the basis functions of the enriched finite element $ \mathcal{S}.$ For this purpose, we introduce the following functions
      \begin{equation*}
e_1:=\lambda_2\lambda_3^2,\qquad e_2:=\lambda_3\lambda_1^2, \qquad e_3:=\lambda_1\lambda_2^2. 
\end{equation*}
\begin{theorem}\label{th2allalf1}
    The basis functions of the enriched finite element $ \mathcal{S}$ have the
    following expressions
\begin{eqnarray}
\mu_i&=&\varphi_i-20\lambda_1\lambda_2\lambda_3,  \qquad i=1,2,3 \label{mui}\\
    \eta_i&=& \phi_i+\frac{5}{3K}\lambda_1\lambda_2\lambda_3, \qquad i=1,2,3, \label{etai}\\
    \zeta_i&=& -\frac{\varphi_i}{12 G}+\frac{K \phi_i}{2G}+\frac{e_i}{G}+ \frac{\lambda_1\lambda_2\lambda_3}{2G}, \qquad i=1,2,3 \label{zetai}
    \\
    \xi&=&\frac{60}{\left\lvert T \right\rvert}\lambda_1\lambda_2\lambda_3.
    \label{zeta}
\end{eqnarray}
Here, $\varphi_i$, $\phi_i$, $K$ and $G$ are defined in~\eqref{varphi},~\eqref{phi},~\eqref{denot} and~\eqref{kprimo}, respectively.      
      \end{theorem}

\begin{proof}
   Without loss on generality, we prove the theorem for $i=1$. 

Since $\mu_1\in \mathbb{P}_3(T)$, it can be expressed as
\begin{equation}\label{mu1eqth}
\mu_1=p_2+a_{1,1}\lambda_2\lambda_3^2+ a_{1,2} \lambda_3\lambda_1^2+a_{1,3}\lambda_1\lambda_2^2 + m_1\lambda_1\lambda_2\lambda_3, 
\end{equation}
where $p_2\in \mathbb{P}_2(T)$ and $a_{1,1}, a_{1,2},a_{1,3}, m_1\in \mathbb{R}$. Applying the linear functional $\mathcal{L}^{\mathrm{enr}}_j$, $j=1,2,3$, to both sides of~\eqref{mu1eqth} and leveraging Lemma~\ref{adomnnewgennew} along with~\eqref{propstar} and~\eqref{cond1}, we derive the linear system
\begin{eqnarray*}
    0=\mathcal{L}_{1}^{\mathrm{enr}}(\mu_1)&=&a_{1,1}G\\
    0=\mathcal{L}_{2}^{\mathrm{enr}}(\mu_1)&=&a_{1,2}G\\
    0=\mathcal{L}_{3}^{\mathrm{enr}}(\mu_1)&=&a_{1,3}G,
\end{eqnarray*}
where $G$ is defined in~\eqref{kprimo}.
Then, we conclude that $a_{1,1}=a_{1,2}=a_{1,3}=0$. Substituting this into~\eqref{mu1eqth}, we get
\begin{equation*}
\mu_1=p_2+m_1\lambda_1\lambda_2\lambda_3.
\end{equation*}
Since, by~\eqref{propstar}, we have
\begin{equation*} \mathcal{I}^{\mathrm{CR}}_j(\lambda_1\lambda_2\lambda_3)=\mathcal{F}^{\mathrm{enr}}_{j}(\lambda_1\lambda_2\lambda_3)=0, \qquad j=1,2,3,
\end{equation*}
the problem reduces to finding a polynomial $p_2\in \mathbb{P}_2(T)$ such that
\begin{equation*}
\mathcal{I}^{\mathrm{CR}}_j(p_2)=\delta_{1j} \qquad \mathcal{F}^{\mathrm{enr}}_{j}(p_2)=0, \qquad j=1,2,3.
\end{equation*}
By Theorem~\ref{th2allalf}, we get
\begin{equation*}  
\mu_1=1-2\lambda_1+m_1\lambda_1\lambda_2\lambda_3=\varphi_1+m_1\lambda_1\lambda_2\lambda_3.
\end{equation*}
To determine the coefficient $m_1$, we use the condition $\mathcal{J}^{\mathrm{enr}}(\mu_1)=0$. Then, by using~\eqref{idbc} of Lemma~\ref{lem1old}, we have
\begin{equation*}   0=\mathcal{J}^{\mathrm{enr}}(\mu_1)=\left\lvert T\right\rvert \left(\frac{20+m_1}{60}\right).
\end{equation*}
Consequently $m_{1}=-20$.

Now we prove the expression of $\eta_1$. Since $\eta_1\in \mathbb{P}_3(T)$, it can be expressed as 
\begin{equation}\label{eta1funth}
\eta_1=p_2+b_{1,1}\lambda_2\lambda_3^2+ b_{1,2} \lambda_3\lambda_1^2+b_{1,3}\lambda_1\lambda_2^2 + m_2\lambda_1\lambda_2\lambda_3,
\end{equation}
where $p_2\in \mathbb{P}_2(T)$ and $b_{1,1},b_{1,2},b_{1,3}, m_2\in \mathbb{R}$.
Applying the linear functional $\mathcal{L}^{\mathrm{enr}}_j$, $j=1,2,3$, to both sides of~\eqref{eta1funth} and leveraging Lemma~\ref{adomnnewgennew} along with~\eqref{propstar} and~\eqref{cond1}, we derive the linear system
\begin{eqnarray*}
    0=\mathcal{L}_{1}^{\mathrm{enr}}(\eta_1)&=&b_{1,1}G\\
    0=\mathcal{L}_{2}^{\mathrm{enr}}(\eta_1)&=&b_{1,2}G\\
    0=\mathcal{L}_{3}^{\mathrm{enr}}(\eta_1)&=&b_{1,3}G,
\end{eqnarray*}
where $G$ is defined in~\eqref{kprimo}.
Then, we conclude that $b_{1,1}=b_{1,2}=b_{1,3}=0$. Substituting this into~\eqref{eta1funth}, we get
\begin{equation*}
\eta_1=p_2+m_2\lambda_1\lambda_2\lambda_3.
\end{equation*}
Since, by~\eqref{propstar}, we have
\begin{equation*} \mathcal{I}^{\mathrm{CR}}_j(\lambda_1\lambda_2\lambda_3)=\mathcal{F}^{\mathrm{enr}}_{j}(\lambda_1\lambda_2\lambda_3)=0, \qquad j=1,2,3,
\end{equation*}
the problem reduces to finding a polynomial $p_2\in \mathbb{P}_2(T)$ such that
\begin{equation*}
\mathcal{I}^{\mathrm{CR}}_j(p_2)=0 \qquad \mathcal{F}^{\mathrm{enr}}_{j}(p_2)=\delta_{1j}, \qquad j=1,2,3.
\end{equation*}
By Theorem~\ref{th2allalf}, we get
\begin{equation*}  
\eta_1=\frac{2\lambda_1-1}{3K}+\frac{1}{2K} (-\lambda_{1}^2 +\lambda_{2}^2+\lambda_{3}^2)+m_2\lambda_1\lambda_2\lambda_3=\phi_1+m_2\lambda_1\lambda_2\lambda_3.
\end{equation*}
To determine the coefficient $m_2$, we use the condition $\mathcal{J}^{\mathrm{enr}}(\eta_1)=0$. Then, by using~\eqref{idbc} of Lemma~\ref{lem1old}, we have
\begin{equation*}    0=\mathcal{J}^{\mathrm{enr}}(\eta_1)=\left\lvert T \right\rvert\left(-\frac{1}{36K}+\frac{m_2}{60}\right),
\end{equation*}
and therefore
\begin{equation*}
m_2=\frac{5}{3K}.    
\end{equation*}

Now we prove the expression of $\zeta_1$. Since $\zeta_1\in \mathbb{P}_3(T)$, it can be expressed as 
\begin{equation}\label{zetafth}
\zeta_1=p_2+c_{1,1}\lambda_2\lambda_3^2+ c_{1,2} \lambda_3\lambda_1^2+c_{1,3}\lambda_1\lambda_2^2 + m_3\lambda_1\lambda_2\lambda_3,
\end{equation}
where $p_2\in \mathbb{P}_2(T)$ and $c_{1,1}, c_{1,2},c_{1,3}, m_3\in \mathbb{R}$. 
Applying the linear functional $\mathcal{L}^{\mathrm{enr}}_j$, $j=1,2,3$, to both sides of~\eqref{mu1eqth} and leveraging Lemma~\ref{adomnnewgennew} along with~\eqref{propstar} and~\eqref{cond1}, we derive the linear system
\begin{eqnarray*}
    1=\mathcal{L}_{1}^{\mathrm{enr}}(\zeta_1)&=&c_{1,1}G\\
    0=\mathcal{L}_{2}^{\mathrm{enr}}(\zeta_1)&=&c_{1,2}G\\
    0=\mathcal{L}_{3}^{\mathrm{enr}}(\zeta_1)&=&c_{1,3}G,
\end{eqnarray*}
where $G$ is defined in~\eqref{kprimo}.
Then, we conclude that 
\begin{equation*}
    c_{1,1}=\frac{1}{G}, \quad c_{1,2}=c_{1,3}=0.
\end{equation*}
Substituting this into~\eqref{zetafth}, we get
\begin{equation*}
\zeta_1=p_2+ \frac{1}{G}\lambda_2\lambda_3^2+m_3\lambda_1\lambda_2\lambda_3.
\end{equation*}
Through straightforward computations, it becomes evident that the conditions
\begin{equation*}
\mathcal{I}^{\mathrm{CR}}_j(\zeta_1)=\mathcal{F}^{\mathrm{enr}}_j(\zeta_1)=0, \qquad j=1,2,3,
\end{equation*}
imply that
\begin{equation*}
    p_2=-\frac{\varphi_1}{12 G}+\phi_1\frac{K}{2G}.
\end{equation*}
Since $\mathcal{J}^{\mathrm{enr}}(\zeta_1)=0$, by~\eqref{idbc}, we obtain
\begin{eqnarray*}
    0&=&\mathcal{J}^{\mathrm{enr}}(\zeta_1)=\frac{\left\lvert T \right\rvert}{4G}\left(\frac{-2+4Gm_3}{60}\right),
\end{eqnarray*}
and therefore 
\begin{equation*}
    m_3=\frac{1}{2G}.
\end{equation*}

It remains to prove the expression of $\xi$. Since $\xi\in \mathbb{P}_3(T)$, it can be expressed as 
\begin{equation}\label{xilastf}
\xi=p_2+d_{1,1}\lambda_2\lambda_3^2+ d_{1,2} \lambda_3\lambda_1^2+d_{1,3}\lambda_1\lambda_2^2 + m_4\lambda_1\lambda_2\lambda_3,
\end{equation}
where $p_2\in \mathbb{P}_2(T)$ and $d_{1,1}, d_{1,2},d_{1,3}, m_4\in \mathbb{R}$. 
Applying the linear functional $\mathcal{L}^{\mathrm{enr}}_j$, $j=1,2,3$, to both sides of~\eqref{xilastf} and leveraging Lemma~\ref{adomnnewgennew} along with~\eqref{propstar} and~\eqref{cond1}, we derive the linear system
\begin{eqnarray*}
   0=\mathcal{L}_{1}^{\mathrm{enr}}(\xi)&=&d_{1,1}G\\
    0=\mathcal{L}_{2}^{\mathrm{enr}}(\xi)&=&d_{1,2}G\\
    0=\mathcal{L}_{3}^{\mathrm{enr}}(\xi)&=&d_{1,3}G,
\end{eqnarray*}
where $G$ is defined in~\eqref{kprimo}.
Then $d_{1,1}=d_{1,2}=d_{1,3}=0$. Substituting this into~\eqref{xilastf}, we get
\begin{equation*}
    \xi=p_2+m_4 \lambda_1\lambda_2\lambda_3.
\end{equation*}
Since
\begin{equation*} \mathcal{I}^{\mathrm{CR}}_j(\lambda_1\lambda_2\lambda_3)=\mathcal{F}^{\mathrm{enr}}_{j}(\lambda_1\lambda_2\lambda_3)=0, \qquad j=1,2,3,
\end{equation*}
the problem reduces to finding a polynomial $p_2\in \mathbb{P}_2(T)$ such that
\begin{equation*}
\mathcal{I}^{\mathrm{CR}}_j(p_2)=0 \qquad \mathcal{F}^{\mathrm{enr}}_{j}(p_2)=0, \qquad j=1,2,3.
\end{equation*}
By Theorem~\ref{th2allalf}, we get $p_2=0$, and then
\begin{equation*}
\xi=m_4\lambda_1\lambda_2\lambda_3.
\end{equation*}
To determine the parameter $m_4$, we use $\mathcal{J}^{\mathrm{enr}}(\xi)=1$. Then, by using~\eqref{idbc} of Lemma~\ref{lem1old}, we have
\begin{equation*}
    1=\mathcal{J}^{\mathrm{enr}}(\xi)=m_4\frac{\left\lvert T \right\rvert}{60},
\end{equation*}
and therefore
\begin{equation*}
m_4=\frac{60}{\left\lvert T \right\rvert}.    
\end{equation*}
Analogously, the theorem can be proven for $i=2$ and $i=3$, and thus, the thesis follows.
\end{proof}  

\begin{theorem}
 The approximation operator relative to the enriched finite element $\mathcal{S}$
\begin{equation}\label{pilinch9}
\begin{array}{rcl}
{\Pi}_3^{{\mathrm{enr}}}: C(T) &\rightarrow& \mathbb{P}_3(T)
\\
f &\mapsto& \displaystyle{\sum_{j=1}^{3} \mathcal{I}^{\mathrm{CR}}_j(f)\mu_j + \sum_{j=1}^{3}\mathcal{F}^{\mathrm{enr}}_j(f)\eta_j+ \sum_{j=1}^{3}\mathcal{L}^{\mathrm{enr}}_j(f)\zeta_j}+ \mathcal{J}^{\mathrm{enr}}(f)\xi,
\end{array}
\end{equation}
reproduces all polynomials of $\mathbb{P}_3(T)$ and satisfies 
\begin{eqnarray*}
\mathcal{I}^{\mathrm{CR}}_j\left({\Pi}_3^{\mathrm{enr}}[f]\right)&=&\mathcal{I}^{\mathrm{CR}}_j(f), \qquad j=1,2,3, \\ \mathcal{F}^{\mathrm{enr}}_j\left({\Pi}_3^{\mathrm{enr}}[f]\right)&=&\mathcal{F}^{\mathrm{enr}}_j(f), \qquad j=1,2,3,\\ \mathcal{L}^{\mathrm{enr}}_j\left({\Pi}_3^{\mathrm{enr}}[f]\right)&=&\mathcal{L}^{\mathrm{enr}}_j(f), \qquad j=1,2,3,\\
\mathcal{J}^{\mathrm{enr}}\left({\Pi}_3^{\mathrm{enr}}[f]\right)&=&\mathcal{J}^{\mathrm{enr}}(f).
\end{eqnarray*} 
\end{theorem}
\begin{proof}
The proof is a consequence of~\eqref{mui},~\eqref{etai},~\eqref{zetai} and~\eqref{zeta}.
\end{proof}

\section{Numerical experiments}\label{sec4}
In this section, we evaluate the effectiveness of the proposed enrichment strategy through several examples. We compare the accuracy of the approximation, as measured by the $L^1$-norm, produced by the standard Crouzeix--Raviart finite element with that produced by the enriched finite elements proposed here. In particular, we compare it with the quadratic polynomial enrichment of the Crouzeix–Raviart finite element $\mathcal{C}$, defined in~\eqref{P2CRfinielement}, and with the cubic polynomial enrichment of the Crouzeix--Raviart finite element $\mathcal{S}$, defined in~\eqref{tripleS}. 
 To this aim, we consider the following functions
 \begin{equation*}
     f_1(x,y)=\frac{1}{x^2+y^2+8}, \quad f_2(x,y)=e^{x+y}, \quad f_3(x,y)=\cos(2x+y), \quad f_4(x,y)=\sqrt{x^2+y^2+1},
 \end{equation*}
\begin{equation*}
    f_5(x,y)=\frac{\sqrt{64-81((x-0.5)^2+(y-0.5)^2)}}{9}-0.5,
\end{equation*}
 \begin{eqnarray*}
 f_6(x,y)&=&0.75e^{-\frac{(9(x+1)/2-2)^2}{4}-\frac{(9(y+1)/2-2)^2}{4}}
 +0.75e^{-\frac{(9(x+1)/2+1)^2}{49}-\frac{(9(y+1)/2+1)}{10}}\\
&&	+ 0.5e^{-\frac{(9(x+1)/2-7)^2}{4}-\frac{(9(y+1)/2-3)^2}{4}}-0.2e^{-(9(x+1)/2-4)^2-(9(y+1)/2-7)^2}. \notag
\end{eqnarray*}
For each experiment, we use triangulations with $N=33,306,2650,23576$ triangles
(see Figure~\ref{Fig:regulatri}).

\begin{figure}[h]
  \centering
   \includegraphics[width=0.24\textwidth]{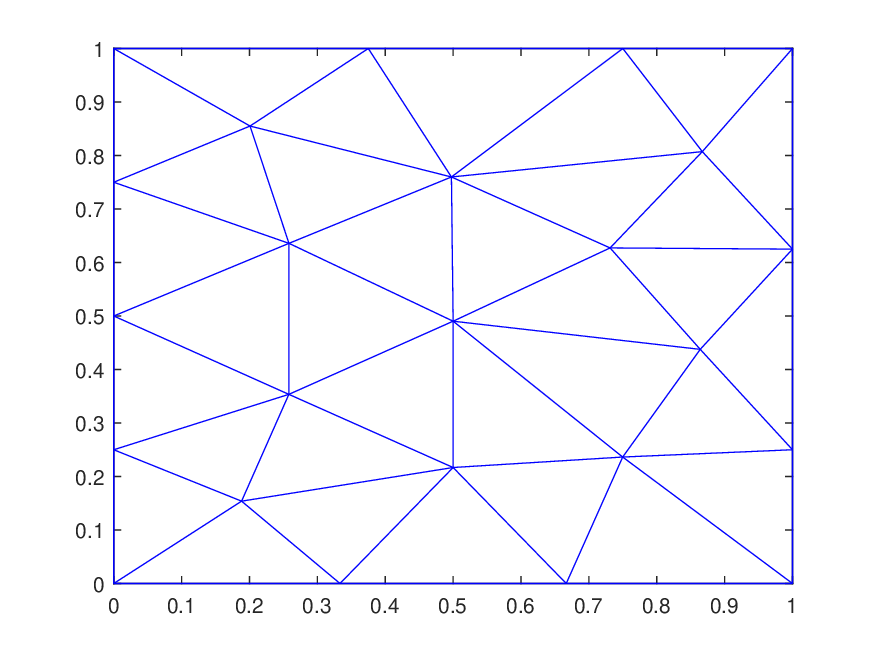} 
    \includegraphics[width=0.24\textwidth]{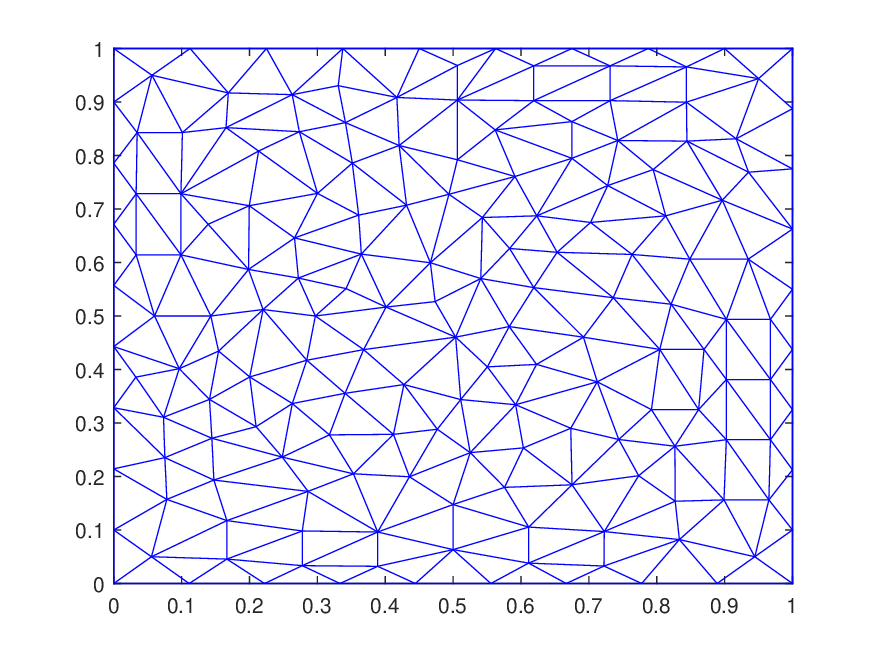} 
        \includegraphics[width=0.24\textwidth]{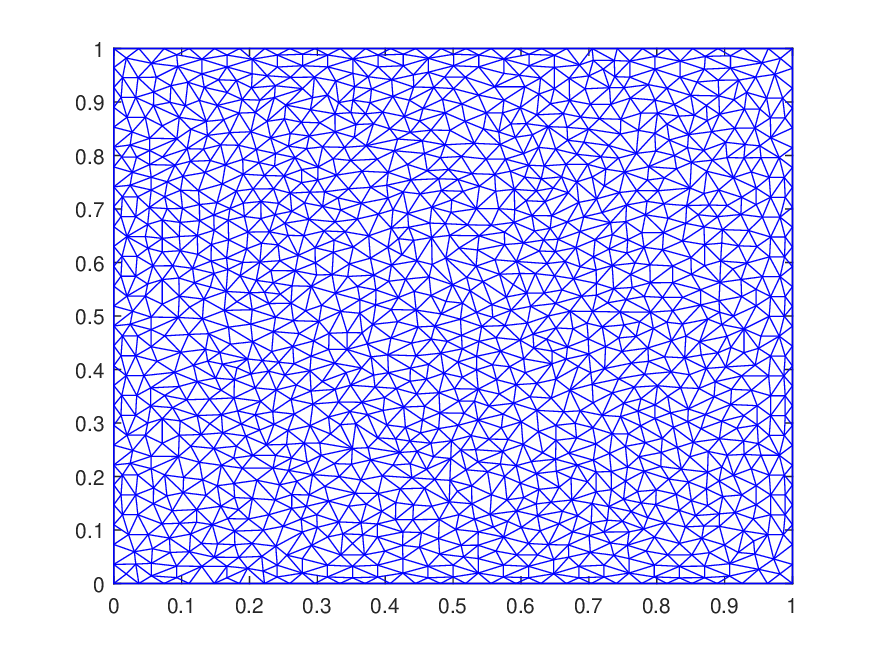}
            \includegraphics[width=0.24\textwidth]{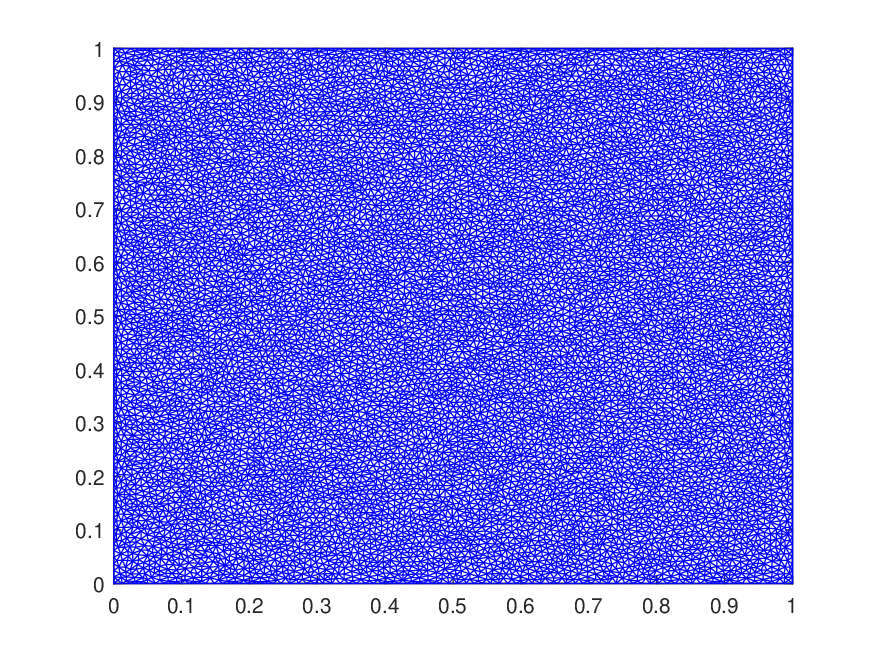}
         \caption{Delaunay triangulation of $N=33$, $N=306$, $N=2650$ and $N=23576$ tringles.}
          \label{Fig:regulatri}
\end{figure}

The numerical experiments are performed using \texttt{MatLab} software. The results are reported in Figure~\ref{fun1}-\ref{fun3}. 
All our results achieved by the enriched finite elements appear much superior to those obtained by the conventional Crouzeix-–Raviart finite element. Notably, the cubic polynomial enrichment produces a more accurate approximation than its quadratic counterpart.
This behavior becomes very satisfactory when the number of triangles in the triangulation increases.

\begin{figure}
  \centering
   \includegraphics[width=0.49\textwidth]{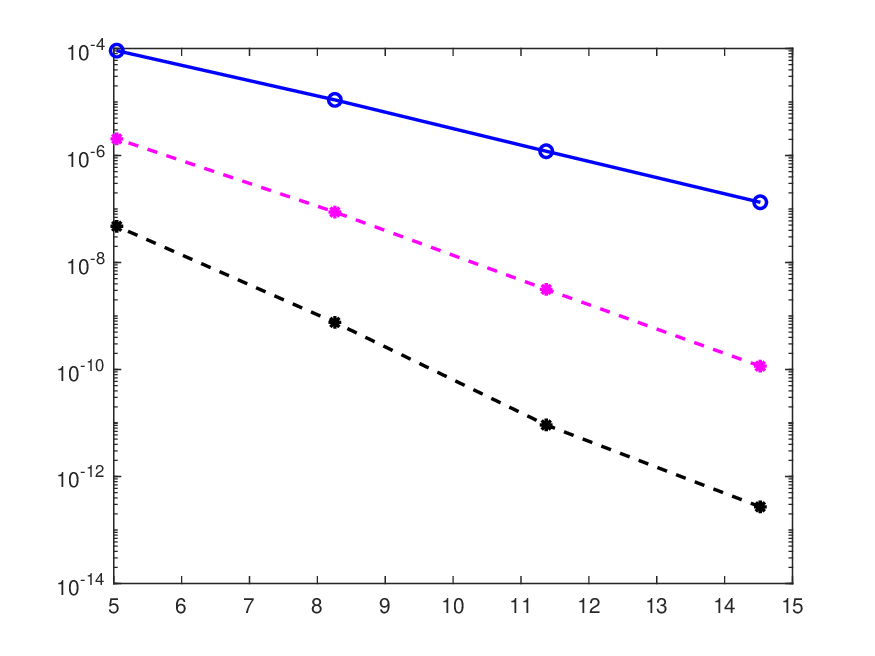} 
    \includegraphics[width=0.49\textwidth]{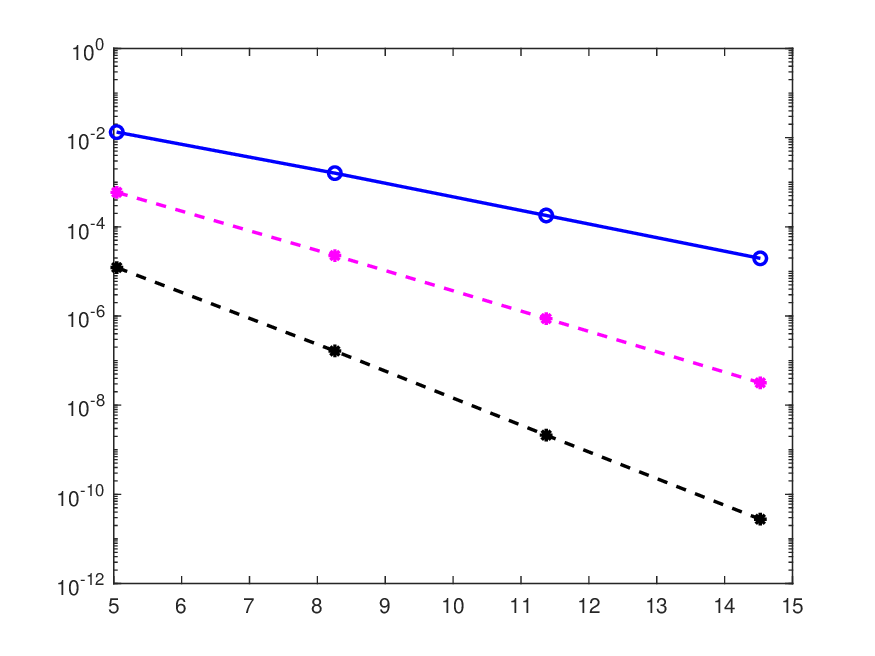} 
         \caption{Loglog plot of errors in $L^1$-norm for approximating the functions $f_1$ (left) and $f_2$ (right). The blue line represents the trend of approximation errors obtained using the standard Crouzeix--Raviart finite element. The magenta line represents the trend of errors obtained with the quadratic enriched finite element $\mathcal{C}$, while the black line represents the trend of errors obtained with the cubic enriched finite element $\mathcal{S}$ with $\alpha=\beta=1$. The comparisons are conducted using Delaunay triangulations from Figure~\ref{Fig:regulatri}.}
          \label{fun1}
\end{figure}

\begin{figure}
  \centering
   \includegraphics[width=0.49\textwidth]{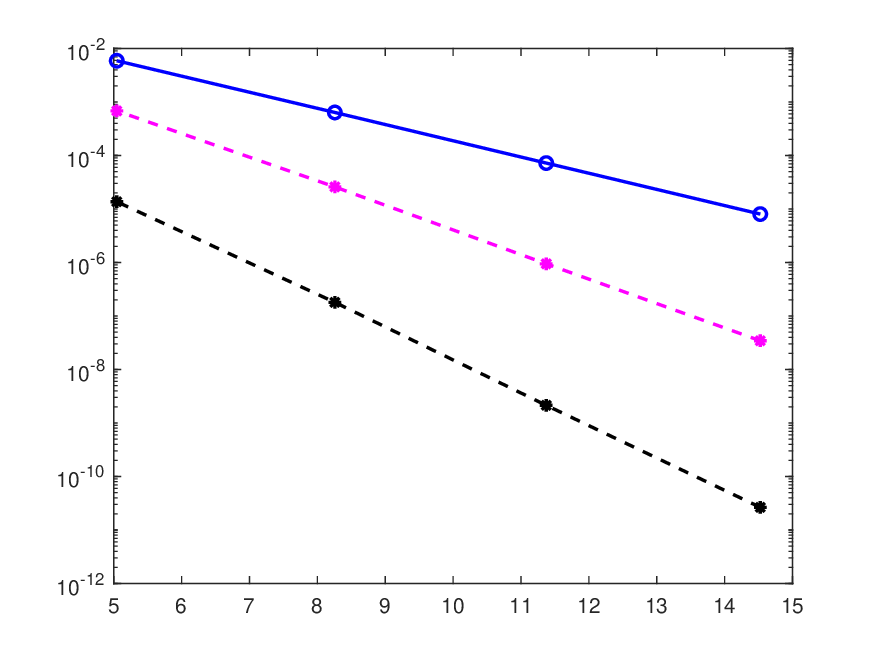} 
    \includegraphics[width=0.49\textwidth]{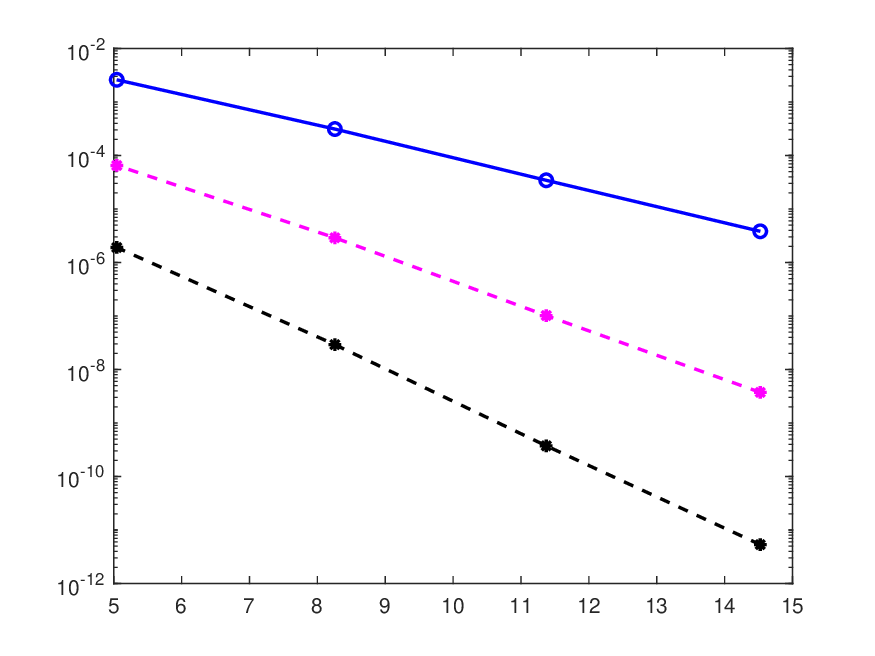} 
         \caption{Loglog plot of errors in $L^1$-norm for approximating the functions $f_3$ (left) and $f_4$ (right). The blue line represents the trend of approximation errors obtained using the standard Crouzeix--Raviart finite element. The magenta line represents the trend of errors obtained with the quadratic enriched finite element $\mathcal{C}$, while the black line represents the trend of errors obtained with the cubic enriched finite element $\mathcal{S}$ with $\alpha=\beta=1$. The comparisons are conducted using Delaunay triangulations from Figure~\ref{Fig:regulatri}.}
          \label{fun2}
\end{figure}

\begin{figure}
  \centering
   \includegraphics[width=0.49\textwidth]{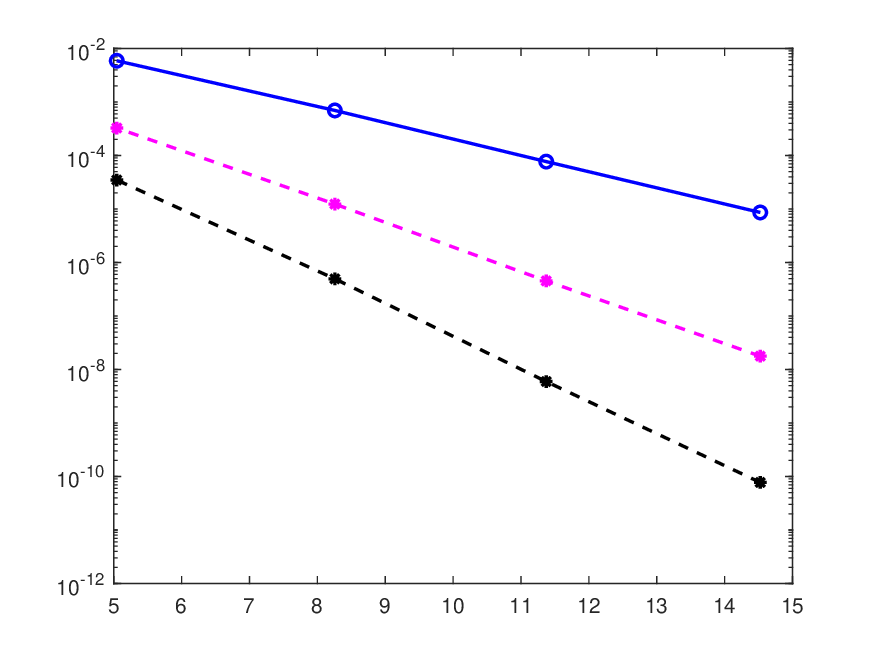} 
    \includegraphics[width=0.49\textwidth]{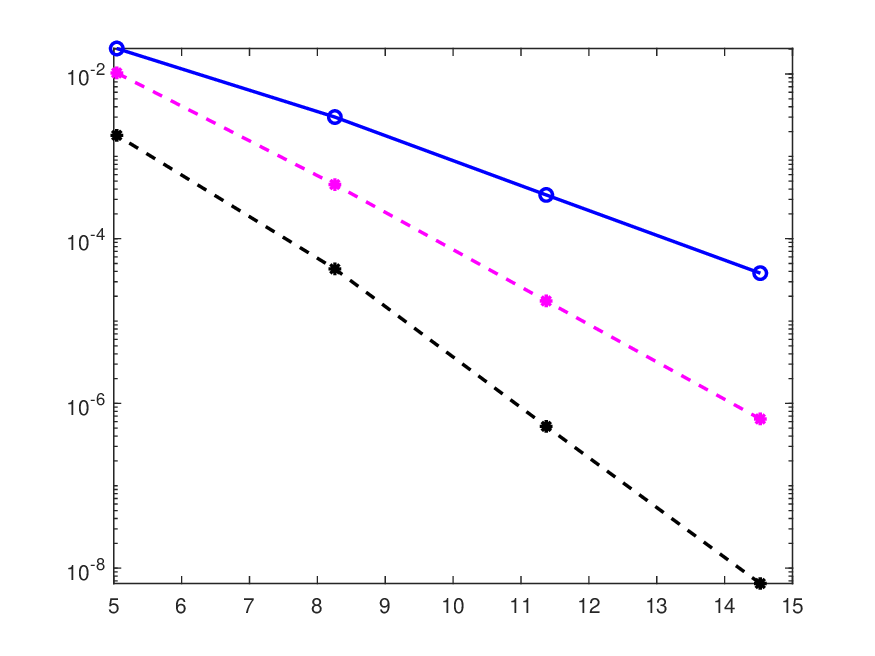} 
         \caption{Loglog plot of errors in $L^1$-norm for approximating the functions $f_5$ (left) and $f_6$ (right). The blue line represents the trend of approximation errors obtained using the standard Crouzeix--Raviart finite element. The magenta line represents the trend of errors obtained with the quadratic enriched finite element $\mathcal{C}$, while the black line represents the trend of errors obtained with the cubic enriched finite element $\mathcal{S}$ with $\alpha=\beta=1$. The comparisons are conducted using Delaunay triangulations from Figure~\ref{Fig:regulatri}.}
          \label{fun3}
\end{figure}

\section*{Acknowledgments}
 This research has been achieved as part of RITA \textquotedblleft Research
 ITalian network on Approximation'' and as part of the UMI group ``Teoria dell'Approssimazione
 e Applicazioni''. The research was supported by GNCS-INdAM 2024 projects and by the grant “Bando Professori visitatori 2022” which has allowed the visit of Prof. Allal Guessab to the Department of Mathematics and Computer Science of the University of Calabria in the spring 2022.

\bibliographystyle{elsarticle-num}
\bibliography{bibliografia.bib}

\end{document}